\documentclass[article]{amsart}
\usepackage[utf8]{inputenc}    
 
\usepackage[french,english]{babel}     
 
\usepackage[cyr]{aeguill}      
 
\usepackage[T1]{fontenc}  
\usepackage{amsmath, amsthm, amsfonts,stmaryrd , amssymb}
\usepackage{mathrsfs}                      
\usepackage{mathtools}                     
\usepackage{booktabs}
\usepackage{todonotes}

\usepackage[utf8]{inputenc}
\usepackage{mathtools}                     
\usepackage{todonotes}
\usepackage{ marvosym }
\usepackage{bbm}                         
\usepackage{framed}
\usepackage[all, ps]{xy}
\usepackage{graphicx}   

\usepackage[colorlinks]{hyperref} 
\hypersetup{
    colorlinks=true,
    linkcolor=blue,  
    urlcolor=cyan,
    pdftitle={RiemannianSplines},
    pdfauthor={R. TAHRAOUI, F.X. VIALARD},
} 
\usepackage{theoremref}

\numberwithin{equation}{section}
\theoremstyle{plain}

\newtheorem{theorem}{Theorem}

\newtheorem{corollary}[theorem]{Corollary}
\newtheorem{proposition}[theorem]{Proposition}

\theoremstyle{definition}
\newtheorem{definition}{Definition}

\def\vol{\on{vol}} 
\def\dvol{\ud\!\vol}

\def\R{{\mathbb R}}

\def\ad{\operatorname{ad}}

\def\Jac{\operatorname{Jac}}

\let\on=\operatorname

\newcommand{\ud}{\,\mathrm{d}}

\let\mc=\mathcal

\newcommand{\eqdef}{\ensuremath{\stackrel{\mbox{\upshape\tiny def.}}{=}}}
\begin{document}
\title[Embedding CH in incompressible Euler]{Embedding Camassa-Holm equations in incompressible Euler
}
\author{Andrea Natale}
\address{INRIA, Project team Mokaplan}
\email{andrea.natale@inria.fr}
\author{Fran\c{c}ois-Xavier Vialard}
\address{Universit\'e Paris-Dauphine, PSL Research University, Ceremade \\ INRIA, Project team Mokaplan}
\email{fxvialard@normalesup.org}

\maketitle

\begin{center}

\end{center}
\begin{abstract}
In this article, we show how to embed the so-called CH2 equations into the geodesic flow of the $H^{\on{div}}$ metric in 2D, which, itself, can be embedded in the incompressible Euler equation of a non compact Riemannian manifold. The method consists in embedding the incompressible Euler equation with a potential term coming from classical mechanics into incompressible Euler of a manifold and seeing the CH2 equation as a particular case of such fluid dynamic equation.
\end{abstract}

\section{Introduction}

The Camassa-Holm (CH) equation as introduced in \cite{CH} is a one dimensional PDE, which is a nonlinear shallow water wave equation \cite{Constantin2008} and is usually written as
\begin{equation}\label{StandardCH}
\partial_{t} u -  \partial_{txx} u +3 \partial_{x} u \,u  - 2\partial_{xx} u \,\partial_x u - \partial_{xxx} u \,u = 0\,.
\end{equation}
This equation has generated a large volume of literature studying its many properties and it has drawn a lot of interest in various communities. Indeed, the physical relevance of this equation is not limited to shallow water dynamics since, for instance, it has been retrieved as a model for the propagation of nonlinear waves in cylindrical hyper-elastic rods \cite{Dai1998}. From the mathematical point of view, it possesses a bi-Hamiltonian structure and can be interpreted as a geodesic flow for an $H^1$ Sobolev metric on the diffeomorphism group on the real line or the circle  \cite{Kouranbaeva1999,Holm1998,Constantin2003}. Of particular interest, peakons (on the real line) are particular solutions of the form 
\begin{equation}\label{EqPeakons}
u(t,x) = \sum_{i = 1}^k p_i(t)e^{-|q_i(t) - x|}\,,
\end{equation}
which can be retrieved as length minimizing solutions for the $H^1(\R)$ metric when the initial and final positions of points $q_i$ are prescribed. On the circle, peakons are described by a slightly different form than \eqref{EqPeakons} where the Green function of $H^1(S_1)$ on the circle replaces the exponential. For particular peakons solutions, blow-up in finite time occurs and it is well-understood, see \cite{SurveyCH}. The blow-up, also called breakdown, of the CH equation has been proposed as a model for wave breaking.
\par
Importantly, the CH equation has been derived in \cite{CH} by an asymptotic expansion from the incompressible Euler equation in the shallow water regime. It is then natural to ask to what extent they differ from the incompressible Euler equations written on a general Riemannian manifold as proposed in \cite{Ebin1970}. In the rest of the paper, all our statements are concerned with strong solutions of the corresponding PDE. Recently, we have proven that the CH equation on $S_1$ is linked with the incompressible Euler equation on $S_1 \times \R_{>0}$. More precisely,
in \cite{VialardGallouetCH}, we prove that the CH equation
\begin{equation} \label{Eq:CH}
\partial_{t} u - \frac14 \partial_{txx} u +3 \partial_{x} u \,u  - \frac12\partial_{xx} u \,\partial_x u - \frac14\partial_{xxx} u \,u = 0\,,
\end{equation}
which describes geodesics on $\on{Diff}(S_1)$ for the right-invariant metric defined by the Sobolev norm $\|u\|^2 + \frac 14 \|u'\|^2$ on the tangent space at identity, can be mapped to particular solutions of the incompressible Euler equation
\begin{equation}\label{Eq:Euler}
\begin{cases}
\dot{v} + \nabla_v v = -\nabla p\,,\\
\nabla \cdot (\rho v) = 0\,,
\end{cases}
\end{equation}
for an appropriately chosen density $\rho(\theta,r)$ on the plane $\R^2 \setminus \{0\}$.
The mapping between the solutions is explicitly given below.
\begin{theorem}[\cite{VialardGallouetCH}]\label{ThEmbedSimple}
Let $u(t,\theta)$ be a smooth solution to the Camassa-Holm equation \eqref{Eq:CH} on $S_1$ then $w: (\theta,r) \mapsto ( u(t,\theta),r \partial_\theta u(t,\theta)/2)$ is a vector field on $S_1 \times \R_{>0}$ and it is a solution of the incompressible Euler equation \eqref{Eq:Euler} for the density $\frac{1}{r^4} r\ud r \ud \theta$.
\end{theorem}
From a Lagrangian point of view, the mapping between the solutions is given by a sort of Madelung transform. Let $\varphi$ be the flow of a smooth solution to the Camassa-Holm equation \eqref{Eq:CH} then $\Psi(\theta,r) \eqdef r \sqrt{\partial_\theta \varphi(\theta)} e^{i\varphi(\theta)}$ is the flow of a solution to the incompressible Euler equation \eqref{Eq:Euler} for the density $\frac{1}{r^4} r\ud r \ud \theta$.
Note that Equation \eqref{Eq:Euler} for this particular density is not exactly the incompressible Euler equation on a Riemannian manifold $(M,g)$ as in \cite{Ebin1970}, which reads
\begin{equation}\label{EqEbinIE}
\begin{cases}
\dot{v} + \nabla_v v = -\nabla p \\
\on{div}_g(v) = 0\,,
\end{cases}
\end{equation}
where $\nabla$ is the Levi-Civita connection associated with the metric $g$ and $\on{div}_g$ denotes the divergence with respect to the volume form associated with $g$.
\par
Motivated by the results in \cite{VialardGallouetCH,TaoUniversality}, the main result of the paper is to embed the CH2 equation, as introduced in \cite{CH2First}, in the incompressible Euler equation. The CH2 equations are a generalization of the CH equation which consists in adding to the CH equation a pressure term depending on an advected density. The CH2 equations read
\begin{equation}\label{EqCH2}
\begin{cases}
\partial_t m + um_x + 2mu_x = -g \rho \partial_x \rho \\
m = u - \partial_{xx}u\\
\partial_t \rho + \partial_x (\rho u) = 0\,,
\end{cases}
\end{equation}
where $\rho(t = 0)$ is a positive density and $g$ is a positive constant.
The embedding is achieved in two steps: first we embed the CH2 equation in a generalized CH equation in 2D (associated with the $H^{\on{div}}$ right-invariant metric) and then we apply the corresponding generalized version of Theorem \ref{ThEmbedSimple}. At this point, it might seem tempting to use the existence of such an embedding to derive new results on one equation thanks to the knowledge of the other. However, the difficulty is pushed in the fact that the Riemannian manifold on which the equations live is often curved and non-complete, see the end of Section \ref{SecTaosTrick}. Nonetheless, this link is interesting from the point of view of classification of fluid dynamic models. The methods we use share some similarities with \cite{Preston2013} by the use of a non right-invariant metric and the Eisenhart lift although our motivation and results differ.
\par
The paper is organized as follows. We present a slight generalization of Tao's embedding of Boussinesq equation into incompressible Euler in Section \ref{SecTao}. Then, in Section \ref{SecHdiv} we recall the geometric arguments explaining Theorem \ref{ThEmbedSimple}, i.e.\ how the geodesic flow of the $H^{\on{div}}$ right-invariant metric can be embedded in the incompressible Euler equation of a cone manifold. Based on these two results, Section \ref{SecCH2embedding}  shows the embedding of the CH2 equation into incompressible Euler. 

\section{Embedding potential dynamics into Euler}\label{SecTao}
This section, based on \cite{TaoUniversality} and Tao's blog, shows how to embed some fluid dynamic equations such as Boussinesq into incompressible Euler. This will be needed in the next section. 
We are concerned with fluid dynamic equations that can be derived from Newton's law with a classical mechanical potential, see \cite[Example 2.2]{ModinMadelung}.
\par
Let $(M,g)$ be a closed Riemannian manifold equipped with the volume form $\vol$, $V: M \mapsto \R_{>0}$ be a smooth positive function on $M$ and $\rho_0$ a smooth positive density on $M$. Consider the following Lagrangian on $\on{SDiff}(M)$, the group of volume preserving diffeomorphisms of $M$,
\begin{equation}\label{EqPotentialFluidDynamic}
\mathcal{L}(\varphi,\dot{\varphi}) = \int_M \frac 12 \|\dot{\varphi}(x)\|_{(g\circ \varphi)(x)}^2  \, \dvol(x) - \int_M V(\varphi) \rho_0(x) \,\dvol(x) \,,
\end{equation}
where $\dvol$ denotes the measure associated with $\vol$ and $ \|\dot{\varphi}(x)\|_{(g\circ \varphi)(x)}^2$ is the same as $g(\varphi(x))(\dot{\varphi}(x),\dot{\varphi}(x))$. The Euler-Lagrange equation for this Lagrangian is 
\begin{equation}\label{EqParticleDynamic}
\frac{D \dot{\varphi}}{Dt}  = - \nabla V(\varphi) \rho_0 - \nabla p \circ \varphi
\end{equation}
where $\frac{{D}}{{D}t}$ is the covariant derivative associated with the metric $g$ and where the pressure $p$ is a time dependent function which accounts for the incompressibility constraint. This equation, rewritten in Eulerian coordinates, introducing the velocity field $u = \dot{\varphi} \circ \varphi^{-1}$, gives
\begin{equation}\label{EqPotential}
\begin{cases}
&\partial_t u + \nabla_u u = - \nabla V \rho - \nabla p\\
& \on{div}_g(u) = 0\\
& \partial_t \rho + \on{div}_g(\rho u) = 0\,.
\end{cases}
\end{equation}
The last equation can also be written as an advection equation $\partial_t \rho + \langle \nabla \rho, u\rangle = 0$ since $\on{div}_g(u) = 0$. A particular case of this formulation is the Boussinesq equation when the potential is due to gravity. Our goal is to see the solutions of System \eqref{EqPotential} as particular solutions of the incompressible Euler equation \eqref{EqEbinIE} on a possibly curved Riemannian manifold. A possible way to achieve this is to formulate the equivalence at the Lagrangian level. The first step consists in interpreting the Hamiltonian evolution of a particle in the presence of an external force as a kinetic evolution with force. This is done in the next paragraph.
\par
\textbf{Eisenhart lift in classical mechanics.}
Since the work of Maupertuis and Jacobi, the fact that a potential evolution can be seen as a reparametrization of a geodesic flow is well-known. By potential dynamic, we mean  solutions of the equation
\begin{equation}\label{EqPotentialDynamic}
\ddot{x} = - \nabla V(x)\,.
\end{equation}
Indeed, let us consider a conformal change of a Riemannian metric $g$ into $e^{2\lambda}g$, where $\lambda$ is a function on $M$. Then, it is possible to find $\lambda$ in terms of $V$ such that the geodesic flow of $e^{2\lambda}g$, describes, \emph{up to a time reparametrization}, the potential dynamic. With the aim of embedding fluid dynamic equations into incompressible Euler, the metric has to be independent of the particle label. Therefore, the Jacobi-Maupertuis transform is not suitable for our purpose. 
\par
Eisenhart in 1929 \cite{Eisenhart} introduced a lifting procedure to describe the potential dynamic as the projection of geodesics on a Riemannian manifold of one additional dimension. Introducing $M \times S_1$ (the second factor can be $\R$ as well), consider the Riemannian metric $\tilde{g}(x,z) = g(x) + \frac{1}{V(x)} (\ud z)^2$. Then, the geodesic equations, written in Hamiltonian form, read
\begin{equation}\label{EqEisenhartHamiltonianLift}
\begin{cases}
\dot{p} = -\partial_x H(p,x) - \frac{1}{2}n^2\partial_x V(x)\\
\dot{n} = 0 \\
\dot{x} = \partial_p H(p,x) \\
\dot{z} = \frac{1}{V(x)} n
\end{cases}
\end{equation}
where $H(p,x) = \frac{1}{2} \langle p,g^{-1}(x)p \rangle$. This system says that the couple $(p,x)$ follows the potential dynamic equation, provided that the constant of the motion is chosen as $n = \sqrt{2}$; Equation \eqref{EqPotentialDynamic} is obtained from System \eqref{EqEisenhartHamiltonianLift} by multiplying the first equation by the cometric $g^{-1}(x)$. We refer to Proposition \ref{ThEisenhartOnRiemannianSide} in Appendix \ref{SecWarpedMetrics} for more details on the geodesic equations associated with this new metric, which is a particular case of warped metrics (see Definition \ref{ThWarped}). 
\par
In order to be sufficiently general, let us remark that this lifting procedure also works when introducing an additional force to the potential dynamic. Indeed, $n$ will still be a constant of the motion. Thus, we have
\begin{proposition}[Eisenhart lift]\label{ThEisenhart}
Let $(M,g)$ be a Riemannian manifold, $V$ be a potential and $F$ be a vector field on $M$ possibly time dependent. Then, the solutions to the ODE system
\begin{equation}
\nabla_{\dot{x}}{\dot{x}} = -\nabla V(x) + F
\end{equation}
are projections on the first variable of the forced geodesic flow on $M \times S_1$ endowed with $\tilde{g}(x,z) =  g(x) + \frac{1}{2V(x)} (\ud z)^2$,
\begin{equation}
\nabla^{\tilde{g}}_{(\dot{x},\dot{z})}{(\dot{x},\dot{z})} = (F,0)\,.
\end{equation}
\end{proposition}
Note that a multiplicative constant on the potential can be taken care of in the constant of the motion which is $n$, the momentum variable associated with $z$.
\par
As a consequence, we have that solutions to Equation \eqref{EqParticleDynamic} correspond to projections on the first variable of $\psi = (\varphi,\lambda)$, which satisfies the equation
\begin{equation}\label{EqRewriting}
\nabla^{\tilde{g}}_{\dot{\psi}} \dot{\psi} = - \nabla^{\tilde{g}}(p,0)\,,
\end{equation}
and the constraint is $\varphi \in \on{SDiff}(M)$. In order to understand Equation \eqref{EqRewriting} as an incompressible Euler equation, the first step is to check that a density is preserved by the flow $\psi$ on $M \times S_1$. Note that $\lambda$ is completely determined by its initial value $\dot{\lambda}(x,y) = V(\varphi(x))\sqrt{2} \rho_0(x)$ which only depends on $x$. Thus, $\lambda$ is a rotation in the $y$ variable with a rotation angle which depends on $x$. As a consequence, the volume form $\vol \wedge \ud z$ is preserved. Thus, $\psi$ preserves a density and it solves \eqref{EqIEWithAnotherDensity}, which can be embedded in the incompressible Euler equation, as described in the introduction. Therefore, we have
\begin{proposition}[Tao's blog and \cite{TaoUniversality}]
The solutions to the incompressible fluid dynamic equations with potential \eqref{EqPotential} are projections of particular solutions to the incompressible Euler equation on $M \times S_1 \times S_1$ for the metric $g(x) + \frac{1}{V(x)}(\ud z)^2 +  V(x) (\ud y)^2$.
\end{proposition}

\section{Embedding CH into Euler}\label{SecHdiv}

This section, based on \cite{VialardGallouetCH}, gives a sketch of the geometric arguments for the embedding of the $H^{\on{div}}$ geodesic flow on a closed Riemannian manifold $(M,g)$ and a computational proof of it.
\subsection{The geometric argument}
Consider the group of diffeomorphism $\on{Diff}(M)$ endowed with the right-invariant $H^{\on{div}}$ metric, that is, the norm is defined in Eulerian coordinates on the velocity field $u$ by 
\begin{equation}\label{eq:hdivmetric}
\| u \|^2_{H^{\on{div}}} = \int_M \| u\|^2_g \dvol + \int_M \frac{1}{4}\on{div}_g(u)^2 \dvol\,.
\end{equation}
For a given right-invariant Lagrangian $\mathcal{L}(\varphi,\dot{\varphi})$ on $\on{Diff}(M)$, one can consider the reduced Lagrangian $l(u) = \mathcal{L}(\on{Id},u)$, where $\on{Id}$ is the identity map on $M$. The Euler-Lagrange equation for a Lagrangian $l(u)$ is called Euler-Arnold or Euler-Poincar\'e equation and it reads
\begin{equation}\label{eq:eulerpoincare}
\frac{\ud}{\ud t} \frac{\delta l} {\delta u} + \ad^*_u  \frac{\delta l} {\delta u} = 0\,,
\end{equation}
where ${\delta l}/{\delta u}$ represents a momentum density \cite{holm2005momentum}. 
Taking $l(u) = \|u\|_{H^{\on{div}}}^2$, this equation describes $H^{\on{div}}$ geodesics on $\on{Diff}(M)$. When $M$ is one dimensional with the Lebesgue measure, it coincides with the Camassa-Holm equation. Note that the Ebin and Marsden framework that proves local well-posedness does not apply in dimension greater than one, since the differential operator associated with the $H^{\on{div}}$ metric is not elliptic. Local well-posedness has been proven in some particular cases, for instance the Euclidean space in \cite{Michor2013} and probably holds in a more general setting such as a closed Riemannian manifold.
\par
The key point to embed the $H^{\on{div}}$ geodesic flow in an incompressible Euler equation consists in viewing $(\on{Diff}(M),H^{\on{div}})$ as an isotropy subgroup isometrically embedded in a larger group on which there is a (non-invariant) $L^2$ metric. 
Such a situation is well-known for the incompressible Euler equation. Indeed, on $\on{Diff}(\R^d)$ it is possible to consider the flat $L^2$ metric with respect to a given volume form, and the subgroup of volume preserving diffeomorphisms $\on{SDiff}(\R^d)$ endowed with this $L^2$ metric. On $\on{SDiff}(\R^d)$, the $L^2$ metric is now right-invariant, see \cite{KhesinGafa} for more details. 
\par
Let us consider the automorphism group of half-densities on $M$, which is a trivial principal bundle once a reference volume measure, in this case $\dvol$, has been chosen. In such a trivialization, the half-densities fiber bundle is $M \times \R_{>0}$ and its automorphism group $\on{Aut}(M \times \R_{>0})$ can be identified with the semi-direct product of groups $\on{Diff}(M)\ltimes C^{\infty}(M,\R_{>0})$ acting on the left on the space of densities by,
\begin{equation}\label{EqLeftAction}
(\varphi,\lambda) \cdot \rho \eqdef \varphi_*(\lambda^2 \rho)\,,
\end{equation}
where $\rho \in \on{Dens}(M)$ is a density and $(\varphi,\lambda) \in \on{Diff}(M)\ltimes C^{\infty}(M,\R_{>0})$.
Note that the semi-direct group composition law is completely defined by the fact that \eqref{EqLeftAction} is a left action. Consider the isotropy subgroup associated with the constant uniform density, that is the subgroup $\on{Aut}_{\vol}(M\times \R_{>0}) = \{ (\varphi,\sqrt{\on{Jac}(\varphi)})  \, | \, \varphi \in \on{Diff}(M) \}$. Mimicking the case of incompressible Euler, we are looking for a right-invariant metric on $\on{Aut}_{\vol}(M\times \R_{>0})$, thus it is completely defined at identity. The derivative of a curve in $\on{Aut}_{\vol}(M\times \R_{>0})$ at identity is $(u,\on{div}_g(u)/2)$ and the $H^{\on{div}}$ metric in \eqref{eq:hdivmetric} is the $L^2$ metric on the product space. We extend this metric on the whole automorphism group
by writing $(u,\on{div}(u)/2) = (\dot{\varphi} \circ \varphi^{-1},\frac{\dot{\lambda}}{\lambda} \circ \varphi^{-1})$. Then the $H^{\on{div}}$ norm yields the following Lagrangian on $\on{Diff}(M)\ltimes C^{\infty}(M,\R_{>0})$
\begin{equation}\label{EqConeMetric}
\mathcal{L}((\varphi,\lambda),(\dot{\varphi},\dot{\lambda}) ) = \int_M (\lambda^2 \|\dot{\varphi}\|^2_{g\circ\varphi} + \dot{\lambda}^2 ) \dvol\,.
\end{equation}
This formula defines a (non right-invariant) $L^2$ metric on $\on{Diff}(M)\ltimes C^{\infty}(M,\R_{>0})$ which satisfies our requirements. Note that this semi-direct product is included as a set in the space of maps from $M$ into $M\times \R_{>0}$ and that the $L^2$ metric is defined by the volume form $\vol$ on $M$ and the Riemannian metric, given by formula \eqref{EqConeMetric}, $g_0 = r^2g + (\ud r)^2$, which is a cone metric. This cone metric is a particular case of a warped metric (see Appendix \ref{SecWarpedMetrics}), which we will use again in the next section. This metric has an important property: the geodesic flow is preserved by every positive scaling in the $r$ direction.
\par
At this point, we have $\on{Aut}_{\vol}(M\times \R_{>0}) \subset \on{Aut}(M\times \R_{>0})$, which is naturally embedded in  $ \on{Diff}(M \times \R_{>0})$ by the map $(\varphi,\lambda) \to [(x,r) \mapsto (\varphi(x),\lambda(x)r)]$.
Due to the fact that scalings are affine isometries for the cone metric, this embedding preserves geodesics. Thus, for every volume form $\nu$ on $\R_{>0}$ such that $\int_0^\infty r^2 \ud \nu(r) = 1$, the embedding $\on{Aut}(M\times \R_{>0}) \hookrightarrow \on{Diff}(M \times \R_{>0})$ is an isometry for the $L^2$ metric with respect to the volume form $\vol \wedge \, \nu$ on $M \times \R_{>0}$.
\par
In order to conclude, it suffices to check that $\on{Aut}_{\vol}(M\times \R_{>0})$, as a subset of $\on{Diff}(M \times \R_{>0})$, preserves (by pushforward) a density on $M \times \R_{>0}$. A direct computation shows that it is the case for $\rho_0(x,r) = \frac{1}{r^{3+d}} \dvol(x,r) = \frac 1{r^3} \ud r \dvol(x)$, where $d$ denotes the dimension of $M$. Note that $\rho_0$ is not integrable at $0$ and it has infinite mass. 
Importantly, the embedding $\on{Aut}_{\vol}(M\times \R_{>0})  \hookrightarrow \on{SDiff}_{\rho_0}(M \times \R_{>0})$, the subgroup of diffeomorphisms that preserve $\rho_0$, is an isometry. This can be summarized by the following diagram,
\SelectTips{eu}{}     
\setlength{\fboxsep}{-1pt} 
\begin{center}
{\color{white}\framebox{{\color{black}$$ 
 
\xymatrix@=5pt{ 
&&  \ar@{<-_{)}}[dd] \, (\on{Aut}(M\times \R_{>0}),L^2_{\vol,g_0}) \,\ar@{^{(}->}[rr]^{Isom.}  && (\on{Diff}(M \times \R_{>0}),L^2_{\vol \wedge \, \nu,g_0})\, \ar@{<-_{)}}[dd]\,  \\
 && \\
(\on{Diff}(M),H^{\on{div}}) \ar@{->}[rr]^(.4){Isom.} && (\on{Aut}_{\vol}(M\times \R_{>0}),L^2_{\vol,g_0})  \ar@{^{(}->}[rr]^{Isom.}
 && (\on{SDiff}_{\rho_0}(M \times \R_{>0}),L^2_{\vol \wedge \, \nu,g_0})\,.
}
$$
}}} 
\end{center}
At this point, we have reformulated the geodesic flow on $\on{Diff}(M)$ for the $H^{\on{div}}$ metric as a geodesic flow on a Riemannian submanifold of $(\on{Aut}(M\times \R_{>0}),L^2_{\vol,g_0})$, which is embedded in the geodesic flow of $\on{SDiff}_{\rho_0}(M \times \R_{>0})$ for an $L^2$ metric with a volume form which  differs from $\rho_0$. 
\par
This is the result formulated in Theorem \ref{ThEmbedSimple}. Alternatively, we give an Eulerian derivation of the result in Section \ref{SecEulerianViewPoint} and we also give an elementary computational proof in Appendix \ref{SecCHComputational}, which addresses the one dimensional case, for simplicity.

\subsection{A few comments on this embedding}

The blow-up of the CH equation corresponds to a Jacobian $\partial_x \varphi$ which vanishes.
Let us explain it briefly. For a standard peakon-antipeakon collision, that is two peakons moving toward each other at the same speed or more generally a antisymmetric peakon configuration, it is possible to prove that the middle point of the configuration will be fixed by the flow and $\partial_x v \to -\infty$. Using the flow equation $\partial_{tx} \varphi = \partial_x v \circ \varphi \, \partial_x \varphi$ at this middle point implies $\partial_x \varphi \to 0$. This is shown in Figure \ref{FigCollision}, in which are plotted the curves $\theta \mapsto r \sqrt{\partial_\theta \varphi(\theta) }e^{i\varphi(\theta)}$ for two different choices of the radius $r$, in blue and green.

\begin{figure}[!htb]
\minipage{0.50\textwidth}
  \includegraphics[width=\linewidth]{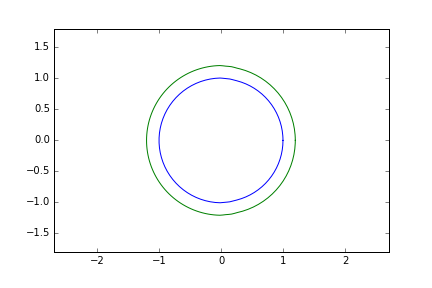}
\endminipage
\minipage{0.50\textwidth}
  \includegraphics[width=\linewidth]{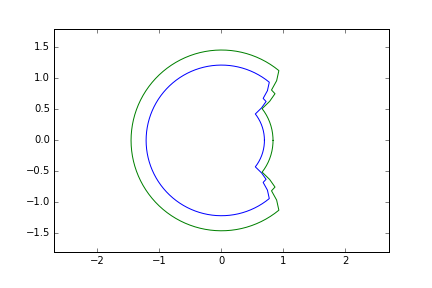}
\endminipage\vfill
\minipage{0.50\textwidth}%
  \includegraphics[width=\linewidth]{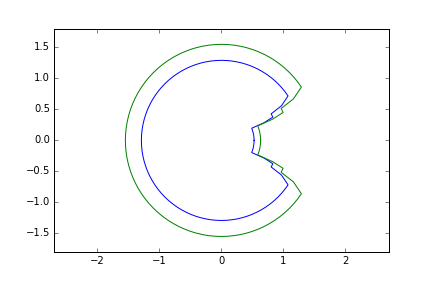}
\endminipage
\minipage{0.50\textwidth}%
  \includegraphics[width=\linewidth]{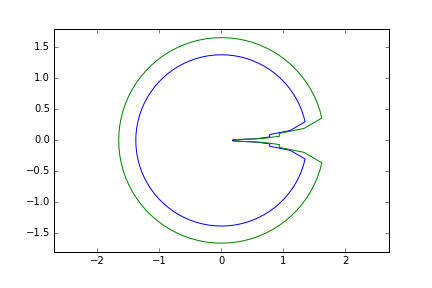}
\endminipage
\caption{A peakons-antipeakons collision represented in the new polar coordinates variables at different time points. The two curves in blue and green are scaled versions of each others: they represent the image under the map $\Psi(\theta,r) = r \sqrt{\partial_\theta \varphi(\theta)} e^{i\varphi(\theta)}$ of two circles on the complex plane with different radii.}
\label{FigCollision}
\end{figure}

\subsection{Applying Tao's change of metric}\label{SecTaosTrick}
In \cite{TaoUniversality}, Tao studied how to map the solutions of some ODEs as particular solutions of the incompressible Euler equation on a Riemannian manifold. As a corollary of his work, it is possible to embed the CH equation in incompressible Euler, at the expense of introducing a new dimension which is introduced to correct the coefficient $\frac{1}{r^{3+d}}$ on the preserved density $\rho_0$. Let us describe it now. Consider $(N,g)$ a Riemannian manifold and the incompressible Euler equation for a particular density $\rho_0$ preserved by the flow, that is
\begin{equation}\label{EqIEWithAnotherDensity}
\begin{cases}
\dot{u} + \nabla_u u = -\nabla p \\
\on{div}_g(\rho_0 u) = 0\,,
\end{cases}
\end{equation}
then, on $M \times S_1$ endowed with the metric $g + \rho_0^2 (\ud y)^2$ and defining $v = (u,0)$, one gets a (particular) solution to the incompressible Euler equation \eqref{EqEbinIE}.
Thus, we have
\begin{proposition}\label{ThCHIntoEuler}
Solutions to the $H^{\on{div}}$ Euler-Arnold equations are particular solutions to the incompressible Euler equation on $M \times \R_{>0} \times S_1$ for the metric $r^2g(x) + (\ud r)^2 +  {r^{-2(3+d)}} (\ud y)^2$, where $(x,r,y) \in M \times \R_{>0} \times S_1$.
\end{proposition}

\begin{corollary}
The incompressible Euler equation \eqref{EqEbinIE} on the Riemannian manifold $S_1 \times \R_{>0} \times S_1$ has particular solutions that exhibit a finite time blow-up.
\end{corollary}
Note that this Riemannian manifold has nonpositive sectional curvature since the metric can be written in a warped product formulation $(\ud x)^2 + (\ud y)^2 + \frac{1}{(x^2 + y^2)^4} (\ud z)^2$ for which the formula \eqref{EqCurvature} for the sectional curvature in \cite{BishopWarped} can be applied.
Note that this Riemannian manifold is not complete since $0$ has to be added to the cone $S_1 \times \R_{>0}$ to make it a complete metric space. Indeed, with the cone point added, it is not a smooth manifold any longer and it has infinite curvature, at least formally, at this point.

\subsection{The Eulerian viewpoint} \label{SecEulerianViewPoint}

It is instructive to examine the relation between the $H^{\on{div}}$ and the Euler equations at the Eulerian level by comparing the relative Euler-Poincar\'e equations. We do this here using differential forms for our results to be valid on the manifold $M$; in particular we will make computations using the volume form $\vol$ rather than the measure $\dvol$.

The momentum equation for the incompressible Euler system in \eqref{EqEbinIE} can be written in terms of the velocity 1-form $v^\flat = g(v,\cdot)$ as follows
\begin{equation}\label{eq:eulerforms}
\dot{v}^\flat + \mathcal{L}_v {v^\flat} +\mathrm{d} \left[ p - \frac{1}{2} g(v,v)\right] = 0\, ,
\end{equation}
since $(\nabla_v v)^\flat = \nabla_v v^\flat = \mathcal{L}_v {v^\flat} -\mathrm{d} g(v,v)/2$. In equation  \eqref{eq:eulerforms}, $\ud$ is the exterior derivative and $\mathcal{L}_v$ is the Lie derivative with respect to $v$, which can be expressed using Cartan's formula as $\mathcal{L}_v = i_v \circ \ud + \ud \circ i_v$, where $i_v$ is the contraction operator with respect to $v$. 
Taking the exterior derivative of \eqref{eq:eulerforms}, we get the well-known advection equation for the vorticity 2-form $\omega = \mathrm{d} v^\flat$,
\begin{equation}
\dot{\omega} +  \mathcal{L}_v {\omega} =0\,.
\end{equation}
On the other hand, the Euler -Poincar\'e equation for the $H^{\on{div}}$ Lagrangian can be written as an advection equation for the momentum density
\begin{equation}\label{eq:advmomdens}
\dot{n} \otimes \mathrm{vol} + \mathcal{L}_u (n \otimes  \mathrm{vol}) =0\, ,
\end{equation}
where $n= u^\flat + \frac{1}{4} \mathrm{d} \delta u^\flat$,  which is equivalent to
\begin{equation}
\dot{n}  + \mathcal{L}_u n  - (\delta u^\flat) n =0\, .
\end{equation}
This is because the quantity ${\delta l}/{\delta u}$ in \eqref{eq:eulerpoincare} can be identified with the momentum density $n\otimes \vol$ in which case $\ad^*_u$ is just the Lie derivative operator \cite{holm2005momentum}. Note also that $\delta$ is the adjoint of $\ud$ with respect to inner product defined by $g$ and in particular $\delta u^\flat = -\on{div}_g u$. 
Consider now the vector field $w = (u, r \mathrm{div}_g u /2)$ on $M\times \mathbb{R}_{>0}$ with the cone metric. Then, 
\begin{equation}
w^\flat =  -\frac{1}{2} r \delta u^\flat \mathrm{d} r + r^2 u^\flat  =  - \frac{1}{4} \delta u^\flat \mathrm{d} r^2 + r^2 u^\flat
\end{equation}
where the metric operations applied to $u$ are computed with respect to $g$, whereas the $\flat$ operator applied to $w$ is computed with respect to the cone metric.
The associated vorticity two-form is
\begin{equation}\label{eq:vorticitych}
\begin{aligned}
\mathrm{d} w^\flat &=  - \frac{1}{4}\mathrm{d} \delta u^\flat \wedge \mathrm{d} r^2 +  \mathrm{d} (r^2 u^\flat) \\
&= \frac{1}{4}\mathrm{d}( r^2 \mathrm{d} \delta u^\flat) +  \mathrm{d} (r^2 u^\flat) \\
&= \mathrm{d} (r^2 n)\,.
\end{aligned}
\end{equation}
Even without specifying $n$, we find that whenever $n$ satisfies \eqref{eq:advmomdens}, the 1-form $r^2 n$ is passively advected by $w$. In fact,
\begin{equation}
\begin{aligned}
r^2 \dot{n} + \mathcal{L}_w (r^2 n ) &= r^2 \dot{n} + r^2 \mathcal{L}_w { n } + n \mathcal{L}_w { r^2 }\\
&= r^2 \dot{n} + r^2 \mathcal{L}_w { n } - n r^2 \delta u^\flat=0\,.
\end{aligned}
\end{equation}
This shows explicitly that by adding one dimension we are able to express advection of one form densities as regular advection using an appropriate lifting for 1-forms and vector fields. Then, because of the particular form of the Lagrangian for CH, we have that $\mathrm{d}(r^2 n) = \mathrm{d} w^\flat$ and therefore $w$ satisfies the Euler equation in rotational form. 

Note that in 1D the relation $ \mathrm{d} w^\flat = \mathrm{d}(r^2 n) $ becomes
\begin{equation}
\mathrm{d} w^\flat = 2 r  \mathrm{d} r \wedge n\,, 
\end{equation}
since $\mathrm{d} n =0$. In terms of standard vector calculus, this is equivalent to the observation that the scalar vorticity of the lifted vector field on the cone $w(x,r) = (u, r \partial_x u /2)$ is given by $\on{curl}(w) = 2u - \frac{1}{2}\partial_{xx}u$ which is twice the momentum of the CH equation $m = u- \frac 14 \partial_{xx} u$. Note that this curl does not depend on the $r$ variable.

The volume form preserved by the flow of the $H^{\on{div}}$ equation lifted to the cone is not the one relative to the cone metric. In fact, if this were preserved, its Lie derivative would then be zero; However, denoting by $d$ the dimension of $M$,
\begin{equation}
\begin{aligned}
 \mathcal{L}_{w}(  \mathrm{d} r^{d+1} \wedge \mathrm{vol}) &= \mathcal{L}_{w}(   \mathrm{d} r^{d+1}) \wedge \mathrm{vol} +    \mathrm{d} r^{d+1} \wedge \mathcal{L}_{w}(\mathrm{vol})\\
 &= -\frac{d+3}{2} \,\delta u^\flat\,  \mathrm{d} r^{d+1} \wedge \mathrm{vol} \,.\\
 \end{aligned}
 \end{equation}
On the other hand
\begin{equation}
\begin{aligned}
\mathcal{L}_{w}( r^{-4} \mathrm{d} r^2 \wedge \mathrm{vol}) &= \mathcal{L}_{w}(  r^{-4} \mathrm{d} r^2) \wedge \mathrm{vol} +   r^{-4} \mathrm{d} r^2 \wedge \mathcal{L}_{w}(\mathrm{vol})\\
&= -\mathrm{d}( r^{-2}  {\delta u^\flat}) \wedge \mathrm{vol} - \delta u^\flat\, r^{-4} \mathrm{d} r^2 \wedge \mathrm{vol}=0 \,.\\
\end{aligned}
\end{equation}
From the previous section, we know that this discrepancy (between the metric of the momentum equation and the volume form preserved by the flow) can be fixed by adding an extra dimension and choosing as metric $r^2 g + (\mathrm{d} r)^2 + r^{-2(3+d)} (\mathrm{d} y)^2$ so that the associated volume form is $r^{-3} \mathrm{d} r \wedge \mathrm{vol} \wedge  \mathrm{d}  y$.  Then,  $\tilde{w} = ( w,0)$ still satisfies the vorticity advection equation on the new manifold and now we also have that 
\begin{equation}
\mathcal{L}_{\tilde{w}} (r^{-4} \mathrm{d} r^2 \wedge \mathrm{vol} \wedge  \mathrm{d}  y) =  \mathcal{L}_{w} (r^{-4} \mathrm{d} r^2 \wedge \mathrm{vol}) \wedge  \mathrm{d}  y = 0\,.
\end{equation}


\section{Embedding the CH2 equations into CH}\label{SecCH2embedding}

This section, based on the two previous ones, contains our main result which consists in embedding the CH2 equations into the CH equation in higher dimension. As a consequence, the CH2 equations can be embedded in incompressible Euler as well.
\par
Let us recall that the CH2 equations are an extension of the CH equation to take into account the free surface elevation in its shallow-water interpretation, while preserving integrability properties. The CH2 equations read
\begin{equation}\label{EqCH2}
\begin{cases}
\partial_t m + um_x + 2mu_x = -g \rho \partial_x \rho \\
m = u - \partial_{xx}u\\
\partial_t \rho + \partial_{x}(\rho u) = 0\,,
\end{cases}
\end{equation}
where $\rho(t = 0)$ is a positive density and $g$ is a positive constant.

There are at least two different point of views to introduce these equations. They can be viewed as a geodesic flow on a semi-direct product of groups $\on{Diff}(S_1) \ltimes C^\infty(S_1,\mathbb{R})$, with the additive group structure on $C^\infty(S_1,\mathbb{R})$, where the metric on the tangent space at identity is the $H^1$ metric on the first factor and the $L^2$ metric on the second. Alternatively, they can be obtained by adding a potential term depending on an advected density. 
\subsection{Embedding using the potential term}
Hereafter, we use the point of view of potential flow developed in Section \ref{SecTao} to obtain the embedding. Consider a Lagrangian $l(u)$ defined on Eulerian velocity fields $u$ on the manifold $M$. We add to this a potential term depending on an advected density. In other words, we consider the Lagrangian
\begin{equation}
l_\rho(u)  = l(u) - F(\rho)\,,
\end{equation}
where $F$ is a functional defined on the space of densities, and we take variations under the constraint of the continuity equation $\partial_t \rho + \on{div}_g(\rho u) = 0$. Then, the Euler-Poincaré equations including the advected quantity $\rho$ read
\begin{equation}\label{EqGeneralSystemEPWithPotential}
\begin{cases}
\displaystyle  \frac{\ud}{\ud t} \frac{\delta l}{\delta u}  + \on{ad}_u^*\frac{\delta l}{\delta u} = - \rho \nabla \frac{\delta F}{\delta \rho} \otimes \dvol\,,\\
\partial_t \rho + \on{div}_g(\rho u) = 0\,.
\end{cases}
\end{equation}
The CH2 equations are a particular case of System \eqref{EqGeneralSystemEPWithPotential} when  $l(u) = \|u\|^2_{H^{\on{div}}}$,  $F(\rho) = \frac{1}{2}\int_{M} \rho^2(x) \dvol(x)$ and $M$ is one dimensional. In the following, however, we will not restrict to the one dimensional case and we will refer to such a slight generalization of System \eqref{EqCH2}  as $H^{\on{div}}$2 equations.
\par
Let us write $F(\rho)$ in terms of $(\varphi,\lambda)$ introduced in Section \ref{SecHdiv},
\begin{align*}
F(\rho) =  \frac{1}{2} \int_{S_1} \on{Jac}(\varphi^{-1})^2 (\rho_0 \circ \varphi^{-1}) ^2\dvol &=  \frac{1}{2} \int_{S_1}  \frac{\rho_0^2}{\lambda^4} \circ \varphi^{-1} \dvol\\
&=  \frac{1}{2} \int_{S_1}  \frac{\rho_0^2}{\lambda^2}  \dvol\,.
\end{align*}
To obtain the last formula, we used the constraint $\varphi_*(\lambda^2 \dvol(x)) = \dvol(x)$.
Therefore, the CH2 equations are the Euler-Lagrange equations associated with the following Lagrangian 
\begin{equation}\label{EqCH2Cone}
\mathcal{L}((\varphi,\lambda),(\dot{\varphi},\dot{\lambda})) = \frac{1}{2} \int_M \lambda^2(x)\|\dot{\varphi}(x)\|_{(g\circ\varphi)(x)}^2 + \dot{\lambda}(x)^2 \, \dvol(x) -  \frac{1}{2} \int_{S_1}  \frac{\rho_0^2(x)}{\lambda^2(x)}  \dvol(x)\,,
\end{equation}
under the constraint $\varphi_*(\lambda^2 \dvol(x)) = \dvol(x)$.
This formula is similar to the Lagrangian \eqref{EqPotentialFluidDynamic} and we have
\begin{theorem}
The $H^{\on{div}}$2 equations on $M$ can be embedded in the $H^{\on{div}}$ geodesic flow on $M \times S_1$. More precisely, the solutions to the $H^{\on{div}}$2 equations can be mapped to particular solutions of the $H^{\on{div}}(M\times S_1)$ geodesic flow. 
\end{theorem}
\begin{proof}
The Euler-Lagrange equation associated with \eqref{EqCH2Cone} reads
\begin{equation}
\frac{D}{Dt} (\dot{\varphi},\dot{\lambda}) = - \nabla \Phi + (0,\frac{2\rho_0}{\lambda^3})\,,
\end{equation}
where $\Phi: M \mapsto \R$ is the pressure term, and it can be rewritten using Proposition \ref{ThEisenhart} as
\begin{equation}
\frac{\tilde{D}}{\tilde{D}t} (\dot{\varphi},\dot{\beta},\dot{\lambda}) = - \nabla (\Phi,0)\,,
\end{equation}
where $\frac{\tilde{D}}{\tilde{D}t} $ is the covariant derivative associated with the metric $r^2 \tilde{g} + (\ud r)^2$, where $\tilde{g} = g+ (\ud y)^2$, on $M \times S_1 \times \R_{>0}$.
It is a particular form of the Euler-Lagrange equation for the following Lagrangian
\begin{equation}\label{eq:lagCHbeta}
\mathcal{L}((\tilde{\varphi},\tilde{\lambda}),(\dot{\tilde{\varphi}},\dot{\tilde{\lambda}})) = \frac{1}{2} \int_{M\times S_1} \tilde{\lambda}^2(x,y) \|\dot{\tilde{\varphi}}(x,y)\|_{(\tilde{g}\circ\tilde{\varphi})(x,y)}^2+ \dot{\tilde{\lambda}}(x,y)^2 \, \dvol(x) \otimes \ud y\,,
\end{equation}
where $\tilde{\varphi} = (\varphi, \beta)$ and 
\begin{equation}
\|\dot{\tilde{\varphi}}(x,y)\|_{(\tilde{g}\circ\tilde{\varphi})(x,y)}^2 = \|\dot{{\varphi}}(x,y)\|_{({g}\circ{\varphi})(x,y)}^2 + |\dot{\beta}(x,y)|^2\,.
\end{equation}
In particular, one needs to choose $\varphi(x,y) = \varphi(x)$ and $\beta(x,y) = \beta(x) + y$ and $\lambda(x,y) = \lambda(x)$. Note that for this choice of $(\varphi,\beta,\lambda)$, one has $(\varphi,\beta)_*(\lambda^2 \dvol(x) \otimes  \ud y) = \dvol(x) \otimes  \ud y$ since 
$\varphi_*(\lambda^2 \dvol(x)) = \dvol(x)$ and $\beta$ is a rotation in the $y$ variable. Hence, the Lagrangian in \eqref{eq:lagCHbeta}  corresponds to the $H^{\on{div}}$ action on $M \times S_1$, which completes the proof.

In  order to make explicit the relation between $\beta$ and $\rho_0$, observe that since the constraint is satisfied independently of the form of $\beta$, we can get its evolution equation simply by taking variations in the Lagrangian. Hence, we find that we must have
\begin{equation}\label{eq:beta}
\dot{\beta}(x) = \frac{C(x)}{\lambda^2(x)}\,,
\end{equation}
where $C(x):M\rightarrow \mathbb{R}$ is a given function. This can be determined by comparing the evolution equation for $\lambda$ for the Lagrangians in \eqref{EqCH2Cone} and \eqref{eq:lagCHbeta}. In particular, we easily find that we must set $C(x) = \rho_0(x)$.
\end{proof}
The expression for the Lagrangian in Equation \eqref{EqCH2Cone} suggests that the embedding can alternatively be derived by formulating CH2 as a geodesic flow on a semidirect product of groups.
\subsection{Embedding using semidirect product of groups}
The previous result shows that the subgroup $\on{Diff}(M) \ltimes C^\infty(M,S_1)$ is totally geodesic in $\on{Diff}(M \times S_1)$ endowed with the $H^{\on{div}}$ metric. The CH2 equations can be derived as a geodesic flow on the group $\on{Diff}(M) \ltimes C^\infty(M,S_1)$ for the right-invariant metric defined by $\| u \|_{H^{\on{div}}}^2 + \| \alpha \|^2_{L^2(M)}$, where $(u,\alpha)$ is an element of the tangent space at identity. Let us choose the group law to be $(\varphi,f) \cdot (\psi,g) = (\varphi \circ \psi, f \circ \psi + g)$, where $S_1 = \R / \mathbb{Z}$. Writing this right-invariant metric in Lagrangian coordinates gives
\begin{equation}\label{EqLagrangianCH2}
\begin{aligned}
\mc{L}((\varphi,f),(\dot{\varphi},\dot{f})) = &
\int_M \on{Jac}(\varphi) \| \dot{\varphi} \|_{g\circ{\varphi}}^2 + \left( \frac{\partial_t \sqrt{\Jac(\varphi)}}{ \sqrt{\Jac(\varphi)}} \right)^2 \on{Jac}(\varphi) \dvol \\& + \int | \dot{f}|^2 \on{Jac}(\varphi) \dvol  \,.
\end{aligned}
\end{equation}
The first and last terms can be grouped together to give $\int_M \on{Jac}(\varphi) | (\dot{\varphi},\dot{f}) |^2 \dvol$ where $(\dot{\varphi},\dot{f})$ is an element of the tangent space at $({\varphi},{f})\in\on{Diff}(M) \ltimes C^\infty(M,S_1)$. Importantly, using the additive group structure of $S_1$, $(\varphi,f)$ is naturally identified with $(x,y) \mapsto (\varphi(x),y + f(x))$ which we still denote by $(\varphi,f)$ with a little abuse of notation.
Obviously, we have $\on{Jac}(\varphi,f) = \on{Jac}(\varphi)$ and thus, denoting $\tilde{g} = g+ (\ud y)^2$, one can write Formula \eqref{EqLagrangianCH2} as
\begin{equation}
\begin{aligned}
\mc{L}((\varphi,f),(\dot{\varphi},\dot{f})) = &
\int_{M \times S_1} \on{Jac}(\varphi,f) \| (\dot{\varphi},\dot{f}) \|^2_{\tilde{g}\circ(\varphi,f)}  \dvol \otimes \ud y \\
& + \int_{M \times S_1} \left( \frac{\partial_t \sqrt{\Jac(\varphi,f)}}{ \sqrt{ \Jac(\varphi,f)}} \right)^2 \on{Jac}(\varphi,f) \dvol \otimes \ud y  \,,
\end{aligned}
\end{equation}
which is once again a particular form of the Lagrangian in \eqref{EqConeMetric} but on the automorpshim group $\on{Aut}_{\vol \wedge \ud y} ((M\times S_1) \times \R_{>0})$ (once the isotropy subgroup relation is made explicit).
It thus shows that the inclusion $\on{Diff}(M) \ltimes C^\infty(M,S_1)  \overset{Isom.}{\hookrightarrow} \on{Diff(M \times S_1)}$ is an isometry, where the latter group is endowed with the $H^{\on{div}}$ metric. However, to realize that it is a totally geodesic subgroup, we need to write the metric on the automorphism group $\on{Aut((M  \times S_1) \times \R_{>0}})$ following the point of view of Section \ref{SecHdiv} in order to get back to formulation \eqref{eq:lagCHbeta}.
\par
Coming back to the embedding into the incompressible Euler equation, we have
\begin{corollary}
The CH2 equations on $S_1$ can be embedded in the incompressible Euler equation on $S_1 \times \R_{>0} \times S_1 \times S_1$ with the metric $r^2 (\ud \theta)^2 + (\ud r)^2 + r^2(\ud y)^2 + \frac{1}{r^{10}} (\ud z)^2$.
\end{corollary}

\subsection{The Eulerian viewpoint for $H^{\on{div}}$2} In order to give an Eulerian description of the embedding of $H^{\on{div}}$2 into $H^{\on{div}}$, we start by rewriting the system in \eqref{EqGeneralSystemEPWithPotential} in terms of differential forms. As before, we let $n = u^\flat + \frac{1}{4}\mathrm{d} \delta u^\flat$ so that
\begin{equation}
\begin{cases}
\dot{n} \otimes \mathrm{vol} + \mathcal{L}_u (n\otimes \mathrm{vol}) = - \rho \, \mathrm{d}\rho \otimes \mathrm{vol}\\
\dot{\rho} \, \mathrm{vol} + \mathcal{L}_u(\rho \,  \mathrm{vol}) = 0\,.
\end{cases}
\end{equation}
In view of equation \eqref{eq:beta}, we have that the couple $(u,\rho)$ can be lifted to a vector field $w$ on $M \times S_1$, given by
\begin{equation}
w = (u, \dot{\beta} \circ \varphi^{-1}) = \left( u, \frac{\rho_0}{\mathrm{Jac}(\varphi)} \circ \varphi^{-1} \right) = (u, \rho)
\end{equation}
satisfying the $H^{\on{div}}$ equation, once $M\times S_1$ is equipped with the metric $g + (\mathrm{d} y)^2$. Specifically, we have 
\begin{equation}
w^\flat = u^\flat + \rho \, \mathrm{d} y\,,
\end{equation}
and since $\rho$ is a function on $M$, $\delta w^\flat= \delta u^\flat$, where metric operations applied to $w$ are computed with respect to the metric on $M\times S_1$ whereas those applied to $u$ are computed with respect to the metric on $M$. In terms of vector fields, this just says that the divergence of $w$ on $M\times S_1$ is equal to that of $u$ on $M$ (and considered as a function on $M\times S_1$).
The momentum associated with $w^\flat$ is given by
\begin{equation}
\tilde{n} = w^\flat + \frac{1}{4} \mathrm{d} \delta w^\flat = n + \rho \, \mathrm{d} y \,.
\end{equation}
The embedding tells us that we should have
\begin{equation}
\dot{\tilde{n}}\otimes \tilde{\mathrm{vol}}+ \mathcal{L}_w (\tilde{n}\otimes \tilde{\mathrm{vol}}) = 0 \,, 
\end{equation} 
where $\tilde{\mathrm{vol}} = \mathrm{vol} \wedge \mathrm{d} y$, or equivalenty
\begin{equation}\label{eq:chlifted}
\dot{\tilde{n}}+ \mathcal{L}_w \tilde{n} - (\delta w^\flat) \tilde{n} = 0 \,, 
\end{equation} 
In order to see this, we rewrite equation \eqref{eq:chlifted} in terms of the lifts $\tilde{n} = n + \rho \,\mathrm{d} y$  and
$w = (u , \rho)$ under the only assumption that $\rho$ be independent of $y$.  
We have
\begin{equation}
\begin{aligned}
\mathcal{L}_{w} \tilde{n} &= \mathcal{L}_{w} {n} + \mathcal{L}_{w} \rho \mathrm{d} y \\
&= \mathcal{L}_{u} {n} + \mathrm{d} \rho^2 + i_{w} (\mathrm{d}\rho \wedge \mathrm{d} y)\\
&= \mathcal{L}_{u} {n} + \mathrm{d} \rho^2 - \rho \, \mathrm{d} \rho  +( i_{u} \mathrm{d}\rho) \,\mathrm{d} y \\
&= \mathcal{L}_{u} {n} + \rho \, \mathrm{d}  \rho  +( i_{u} \mathrm{d}\rho) \,\mathrm{d} y \,.
\end{aligned}
\end{equation}
Hence, provided that $(u,n,\rho)$ satisfy equation \eqref{eq:chlifted}, 
\begin{equation}
\begin{aligned}
\dot{\tilde{n}}+ \mathcal{L}_w \tilde{n} - (\delta w^\flat) \tilde{n}  &= \dot{\tilde{n}}+ \mathcal{L}_w \tilde{n} - (\delta u^\flat) \tilde{n} \\
&= \dot{{n}}+ \dot{\rho} \, \mathrm{d}y +  \mathcal{L}_{u} {n} + \rho\, \mathrm{d} \rho +( i_{u} \mathrm{d}\rho) \,\mathrm{d} y  - (\delta u^\flat)
 \tilde{n} \\
&= [\dot{{n}}+  \mathcal{L}_{u} {n} + \rho \, \mathrm{d} \rho - (\delta u^\flat) n] + [ \dot{\rho} \mathrm{d}y  + ( i_{u} \mathrm{d}\rho) \,\mathrm{d} y  - (\delta u^\flat) \rho \, \mathrm{d} y]\\
&=  [\dot{\rho}   + ( i_{u} \mathrm{d}\rho)  - (\delta u^\flat) \rho ] \mathrm{d} y\,.
\end{aligned}
\end{equation}
However, 
\begin{equation}
[\dot{\rho}   + ( i_{u} \mathrm{d}\rho)   - (\delta u^\flat) \rho ] \mathrm{vol} = 
\dot{\rho} \, \mathrm{vol} + \mathcal{L}_u(\rho \,  \mathrm{vol})  = 0\,,
\end{equation}
and therefore equation \eqref{eq:chlifted} holds. In other words, we are able to lift the momentum and velocity variables in a larger space so that these satisfy the $H^{\on{div}}$ equation. Moreover, we find that the connection between momentum and  velocity is invariant under this lift. In fact, we have
\begin{equation}
\tilde{n} =  w^\flat + \frac{1}{4} \mathrm{d} \delta w^\flat 
\end{equation}  
by construction, since this is how we found the lift for $n$.

\section{Further questions}

The question of embedding ODE or PDE equations into incompressible Euler, satisfying some natural requirements, might be a very soft one as indicated by \cite{TaoUniversality}. A natural question to ask is if it is possible to embed more general Euler-Poincaré-Arnold equations, for instance with higher-order norms rather than $H^{\on{div}}$. 
\par Another possible direction, which was partially addressed in this article, consists in strengthening the conditions on the embedding by requiring it to be, for example, isometric in some sense. Of course, this makes sense only when the underlying PDE comes from a variational principle. Let us show informally why this additional condition leads to a more constrained situation. An example of a right-invariant geodesic flow which would be not embeddable into incompressible Euler is the case of the $L^2$ right-invariant metric on $\on{Diff}(M)$. In this case, the induced distance is known to be degenerate \cite{Michor2005} whereas the distance on $\on{SDiff}(N)$ is obviously not degenerate when $N$ is a closed Riemannian manifold.
%
%
%
\section*{Acknowledgments}
The research leading to these results has received funding from the People Programme (Marie
Curie Actions) of the European Union’s Seventh Framework Programme (FP7/2007-2013)
under REA grant agreement n.
PCOFUND-GA-2013-609102, through the PRESTIGE
programme coordinated by Campus France.

The second author would like to thank Franck Sueur for mentioning a possible link between \cite{TaoUniversality} and \cite{VialardGallouetCH}.

\appendix

\section{An elementary proof of the embedding of CH into Euler}\label{SecCHComputational}
Recall again that we write formal computations on sufficiently regular solutions of the CH equation. Existence of such solutions $u(t)$ in $H^s$ when $s>3/2$ are guaranteed by Ebin-Marsden \cite{Ebin1970} until a possible blow-up time. Then, the associated flow $\varphi(t)$ is well defined as the solution of the ODE
\begin{equation}
\begin{cases}
\partial_t \varphi(t,x) = u(t,\varphi(t,x)) \\
\varphi(0,x) = x \text{  } \forall x \in S_1\,.
\end{cases}
\end{equation}
We now compute the Eulerian vector field $v$ associated with $\Psi$. In polar coordinates, we collect a few useful formulas:
\begin{align*}
&\Psi(\theta,r) = (\varphi(\theta),r\sqrt{\partial_x \varphi})\\
&\Psi^{-1}(\theta,r) = \left(\varphi^{-1}(\theta),r\frac{1}{\sqrt{\partial_x \varphi \circ \varphi^{-1}}}\right)\\
&\partial_t \Psi(\theta,r) = \left(\partial_t \varphi, \frac r2 \frac{\partial_{tx} \varphi}{\partial_x \varphi}\right)\\
& \partial_t \Psi \circ \Psi^{-1} = \left(u(\theta),\frac{r}{2} \partial_x u\right)
\end{align*}
Using the usual orthonormal basis on $\R^2$, $e_r = \frac{\partial}{\partial_r}$ and $e_\theta = \frac{1}{r}\frac{\partial}{\partial_\theta}$, we have
\begin{equation}
v = \frac{1}{2}r\partial_\theta u\, e_r + ru(\theta) \,e_\theta\,.
\end{equation}

\par
We write the divergence constraint $\nabla \cdot (\rho v) = 0$ explicitly.
One has
\begin{align*}
\frac{1}{r} \partial_r (r \rho u_r) + \frac{1}{r} \partial_\theta (\rho u_\theta) &= \frac{1}{r} \partial_r \left(\frac{1}{2r^2} \partial_\theta v\right) + \frac{1}{r^4} \partial_\theta v\\
& =-\frac{1}{r^4} \partial_\theta v + \frac{1}{r^4} \partial_\theta v = 0\,.
\end{align*}
\par
The incompressible Euler equation in polar coordinates reads
\begin{equation} \label{Eq:PolarEuler}
\begin{cases}
\partial_t v_r + v_r \partial_r v_r + \frac{v_\theta}{r} \partial_\theta v_r - \frac{1}{r} v_\theta^2= - \partial_r P \\
\partial_t v_\theta + v_r \partial_r v_\theta + \frac{v_\theta}{r} \partial_\theta v_\theta + \frac{1}{r} v_rv_\theta = -\frac{1}{r}\partial_\theta P\,.
\end{cases}
\end{equation}
We then check that these equations are satisfied for $v$ as defined above. In the following, we will use the notation $\partial_x$ for $\partial_\theta$. Although this notation is a bit abusive, it makes a clear difference between the vector field associated with $\varphi$ and the one associated with $\Psi$.
The first equation in System \eqref{Eq:PolarEuler} gives
\begin{align*}
&\frac r2 \partial_{tx} u + \frac r4 (\partial_x u)^2  + \frac r2 u \partial_{xx} u - r u^2 = -\partial_r P\\
&\frac{1}{2} \partial_{tx} u + \frac 14 (\partial_x u)^2  + \frac 12 u \partial_{xx} u -  u^2 = -\frac 1r \partial_r P\,.
\end{align*}
The second equation in System \eqref{Eq:PolarEuler} gives
\begin{align*}
& r \partial_{t} u + \frac{1}{2}ru \partial_x u   + r u\partial_x u + \frac 12 r u\partial_x u  = -\frac 1r \partial_r P\\
& \partial_{t} u + 2 \partial_x u u  = -\frac 1{r^2} \partial_\theta P \,.
\end{align*}
The Euler equations now read
\begin{equation}
\begin{cases}
\frac{1}{2} \partial_{tx} u + \frac 14 (\partial_x u)^2  + \frac 12 u \partial_{xx} u -  u^2 = -\frac 1r \partial_r P \\
\partial_{t} u + 2 \partial_x u u  = -\frac 1{r^2} \partial_\theta P\,.
\end{cases}
\end{equation}
Note that the left-hand side of the two equations only involves the variable $\theta$. It does not depend on $r$. In particular, it implies that $P$ is necessarily of the form
$P(\theta,r) = r^2 p(\theta)$. As a consequence, $-\frac 1r \partial_r P = -2p$ and $ -\frac 1{r^2} \partial_\theta P = -\partial_x p$. Thus, we obtain the following system,
\begin{equation}
\begin{cases}
\frac{1}{2} \partial_{tx} u + \frac 14 (\partial_x u)^2  + \frac 12 u \partial_{xx} u -  u^2 = -2p \\
\partial_{t} u + 2 \partial_x u u  = -\partial_x p \,.
\end{cases}
\end{equation}
This system is verified if and only if there exists $p$ such that the system above is satisfied. Therefore, since $p$ is given by the first equation, $p$ exists if and only if
\begin{equation}
\frac12 \partial_x\left( \frac{1}{2} \partial_{tx} u + \frac 14 (\partial_x u)^2  + \frac 12 u \partial_{xx} u -  u^2 \right) = \partial_{t} u + 2 \partial_x u u\,.
\end{equation}
After expanding all the terms, we get
\begin{equation}
\frac14 \partial_{txx} u + \frac 14 \partial_x u \partial_{xx} u  + \frac 14 \partial_x u \partial_{xx} u +  \frac 14 u \partial_{xxx} u-  u \partial_x u = \partial_{t} u + 2 \partial_x u u\,,
\end{equation}
which gives, after simplification, the Camassa-Holm equation \eqref{Eq:CH}. 
\par
As a last comment, the standard CH equation \eqref{StandardCH}
\begin{equation}
\partial_{t} u -  \partial_{txx} u +3 \partial_{x} u \,u  - 2\partial_{xx} u \,\partial_x u - \partial_{xxx} u \,u = 0\,.
\end{equation}
can be retrieved as a incompressible Euler geodesic flow on the cone for the metric $r^2(\ud \theta)^2 + 4 (\ud r)^2$. Other parameters in the CH equations can be retrieved by varying the parameters of the cone metric, that is, the angel of the cone.

\section{Warped metrics}\label{SecWarpedMetrics}

In this section, we collect the geodesic equations for a warped Riemannian metric and its curvature tensor.
\begin{definition}\label{ThWarped}
Let $(M,g_M)$ and $(N,g_N)$ be two Riemannian manifolds and $w: M \mapsto \R_{>0}$ be a smooth map. The warped metric on $M \times N$ is the Riemannian metric $g_M + w g_N$.
\end{definition}

The simplest example of a warped metric is on $\R_{>0} \times S_1$ with the metric $(\ud r)^2 + r^2 (\ud \theta)^2$, which is the Euclidean metric in polar coordinates on $\R^2 \setminus \{ 0 \}$.

\begin{proposition}
The geodesic equations for the warped metric in Definition \ref{ThWarped} are, for $(x,y) \in M \times N$,
\begin{align}\label{EqGeodesicWarped}
&\nabla_{\dot{x}} \dot{x} - \frac12 g_N(\dot{y},\dot{y}) \nabla w(x) = 0\\
&\nabla_{\dot{y}}{\dot{y}} +  \dot{y} \frac{\ud}{\ud t} \log(w(x(t))) = 0\,.
\end{align}
\end{proposition}

These two equations imply that the geodesic motion on $N$ is a reparametrized geodesic. The motion on $M$ can be rewritten only in terms of $x$ as follows, since
\begin{align}
\frac{\ud}{\ud t} g_N(\dot{y},\dot{y}) &= 2g_N(\nabla_{\dot{y}}\dot{y},\dot{y}) = -2g_N(\dot{y},\dot{y}) \frac{\ud}{\ud t} \log(w(x(t)))
\end{align}
which implies that there exists a constant $c >0$, which is determined by the initial conditions of the geodesic such that
\begin{equation}
g_N(\dot{y},\dot{y}) = \frac{c}{w(x(t))^2}\,.
\end{equation}
Using this equality in the first geodesic equation of system \eqref{EqGeodesicWarped} leads to 
\begin{equation}
\nabla_{\dot{x}} \dot{x} - \frac12 \frac{c}{w(x(t))^2} \nabla w(x) = 0
\end{equation}
which can be rewritten as a potential evolution
\begin{equation}
\nabla_{\dot{x}} \dot{x} =- \frac12 \nabla \frac{c}{w}(x) \,.
\end{equation}
Solutions of this equation are the critical paths for the action $A(x,\dot{x}) = g_M(\dot{x},\dot{x}) - \frac{c}{2w}(x)$. This can be rewritten as follows.
\begin{proposition}\label{ThEisenhartOnRiemannianSide}
Let $V$ be a positive function on $M$, then the solutions of $\nabla_{\dot{x}} \dot{x} = - c \nabla V$ (for every positive constant $c$) are  the projections on $M$ of the geodesic flow on the warped product $M \times N$ for $w = \frac{1}{V}$.
\end{proposition}
The curvature of warped metrics has been computed for instance in \cite{BishopWarped} and the formula is as follows, denoting respectively $K_w,K_M,K_N$, the sectional curvature of $M \times_w N$, $M$ and $N$, and $\langle \cdot, \cdot \rangle_N$ or $| \cdot |_N$ denotes the scalar product or norm given by $g_N$,
\begin{multline} \label{EqCurvature}
K_w(x,y)((u_1,v_1),(u_2,v_2)) = K_M(u_1,u_2) (|u_1|^2 |u_2|^2 - \langle u_1,u_2\rangle^2) \\- w(x) [|v_1|_N^2\nabla^2w(x)(u_2,u_2) + |v_2|_N^2\nabla^2w(x)(u_1,u_1) - 2 \langle v_1,v_2\rangle_N \nabla^2w(x)(u_1,u_2)]\\ + w(x)^2[K_N(v_1,v_2) - |\nabla w(x)|^2](|v_1|_N^2 |v_2|_N^2 - \langle v_1,v_2\rangle_N^2)\,.
\end{multline}
\bibliographystyle{plain}      
\bibliography{articles,SecOrdLandBig,references,sum_of_kernels}   

\end{document}